\documentclass[10pt,a4paper]{amsart}
\usepackage[english]{babel}

\usepackage{amsmath,amssymb,amsthm}

\theoremstyle{plain}
\newtheorem{theorem}{Theorem}[section]
\newtheorem{prop}[theorem]{Proposition}
\theoremstyle{definition}

\newtheorem{remarks}[theorem]{Remarks}
\newtheorem{example}[theorem]{Example}
\newtheorem{definition}[theorem]{Definition}

\newcommand{\R}{\mathbb{R}}
\newcommand{\T}{\mathbb{T}}
\newcommand{\eps}{\varepsilon}
\newcommand{\e}{\mathrm{e}}
\newcommand{\Id}{\mathrm{Id}}

 \DeclareMathOperator{\re}{Re}
 \DeclareMathOperator{\iso}{Iso}

\begin{document}
\begin{center}
J.~Funct.~Anal.\ (2008), \texttt{doi:10.1016/j.jfa.2008.06.004}
\end{center}

\title[Isometries and duality]{{\Large The group of isometries of a Banach space and duality} }
\author[Miguel Mart\'{\i}n]{Miguel Mart\'{\i}n}
\address{Departamento de An\'{a}lisis Matem\'{a}tico, Facultad de Ciencias,
Universidad de Granada, E-18071 Granada, Spain}
\email{mmartins@ugr.es}
 \date{March 4th, 2008}
\thanks{Supported by Spanish MEC project MTM2006-04837, Junta de
Andaluc\'{\i}a grants FQM-185 and FQM-01438.}

\begin{abstract}
We construct an example of a real Banach space whose group of
surjective isometries has no uniformly continuous one-parameter
semigroups, but the group of surjective isometries of its dual
contains infinitely many of them. Other examples concerning
numerical index, hermitian operators and dissipative operators
are also shown.
\end{abstract}

\maketitle

\section{Introduction}
Given a real or complex Banach space $X$, its dual space is denoted by
$X^*$ and the Banach algebra of all bounded linear operators on $X$ by
$L(X)$. If $T\in L(X)$, $T^*\in L(X^*)$ denotes the adjoint operator of
$T$.

To understand the geometry of a Banach space, it is very useful to know
the structure of its \emph{surjective isometries}, i.e.\ surjective linear
applications which preserve the norm. We refer the reader to the recent
books by R.~Fleming and J.~Jamison \cite{Fle-Jam1,Fle-Jam2} and references
therein for background.

In this paper, we would like to investigate the relationship
between the group $\iso(X)$ of all surjective isometries on a
Banach space $X$ and the one of its dual. It is well known that
the map $T\longmapsto (T^*)^{-1}$ is a group monomorphism from
$\iso(X)$ into $\iso(X^*)$. There are situations in which this
application is an isomorphism. Namely, when $X$ is reflexive
(obvious) and, more generally, when $X$ is $M$-embedded
\cite[Proposition~III.2.2]{HWW} and, even more generally, when
there exists a unique norm-one projection
$\pi:X^{***}\longrightarrow X^*$ with $w^*$-closed kernel
\cite[Proposition~VII.1]{Godefroy}. Other condition to get that
all the surjective isometries of $X^*$ are $w^*$-continuous is
to assure that $X$ has a shrinking $1$-unconditional basis
\cite{Legisa-89}. On the other hand, the above map is not
always surjective, i.e.\ it is possible to find surjective
isometries on the dual of a Banach space which are not
$w^*$-continuous (see \cite[pp.~184]{Legisa-89} for an easy
example on $C[0,1]$).

Therefore, in some sense, the group $\iso(X^*)$ is bigger than
$\iso(X)$. Our aim in this paper is to show that this
phenomenon could be even stronger. We construct a Banach space
$X$ such that the geometry of $\iso(X)$ around the identity is
trivial (the tangent space to $\iso(X)$ at $\Id$ is zero),
while the geometry of $\iso(X^*)$ around the identity is as
rich as the one of a Hilbert space (the tangent space to
$\iso(X^*)$ at $\Id$ is infinite-dimensional). By the
\emph{tangent space} of the group of surjective isometries on a
Banach space $Z$ at $\Id$ we mean the set
$$
\mathcal{T}(\iso(Z),\Id)=\left\{f'(0)\ : \ f\!:\![-1,1]\longrightarrow
\iso(Z),\ f(0)=\Id,\ \text{$f$ differentiable at $0$}\right\}.
$$
Equivalently (see Proposition~\ref{prop:Rosenthal}),
$\mathcal{T}(\iso(Z),\Id)$ is the set of the generators of
uniformly continuous one-parameter semigroup of isometries. By
a \emph{uniformly continuous one-parameter semigroup of
surjective isometries} on a Banach space $Z$ we mean a
continuous function $\Phi:\R^+_0\longrightarrow L(Z)$ valued in
$\iso(Z)$ such that $\Phi(t+s)=\Phi(t)\Phi(s)$ for every
$s,t\in \R$; equivalently (see \cite[Chapter IX]{Hil-Phi} for
instance), a group of the form
$$
\left\{\exp(\rho\, T)\ : \ \rho\in\R \right\}
$$
for some $T\in L(Z)$ which is contained in $\iso(Z)$, where $\exp(\cdot)$
denotes the exponential function, i.e.\ $\displaystyle
\exp(T)=\sum_{n=0}^\infty \dfrac{T^n}{n!}\in L(Z)$.

The main aim of this paper is to make a construction which allows us,
among others, to present the following result (see
Example~\ref{main-example}).

\begin{quote}
\noindent{\slshape There exists a real Banach space $X$ such that
$\iso(X)$ does not contain any non-trivial uniformly continuous
one-parameter subgroup, while $\iso(X^*)$ contains infinitely many
different uniformly continuous one-parameter subgroups. Equivalently,
$\mathcal{T}(\iso(X),\Id)=\{0\}$ but $\mathcal{T}(\iso(X^*),\Id)$ is
infinite-dimensional.\ }
\end{quote}

The paper is divided into four sections, including this introduction. The
second one is devoted to explain the main tool we are using, the numerical
range of operators on Banach spaces and its relationship with isometries.
In the third section, for every separable Banach space $E$, we construct a
C-rich subspace $X(E)$ of $C[0,1]$ (see Definition~\ref{def-crich}) such
that $E^*$ is (isometrically isomorphic to) an $L$-summand of $X(E)^*$.
These spaces inherit some properties from $C[0,1]$ while their dual spaces
share properties with the particular spaces $E$. This allows us to
present, in section~$4$, the aforementioned example and other ones
concerning the numerical index, hermitian operators, and dissipative
operators.

We finish the introduction with some notation. Given a real or complex
Banach space $X$, we write $B_X$ for the closed unit ball and $S_X$ for
the unit sphere of $X$. We write $\T$ to denote the set of all modulus-one
scalars, i.e.\ $\T=\{-1,1\}$ in the real case and
$\T=\left\{\e^{i\,\theta}\, : \ \theta\in \R\right\}$ in the complex case.
A closed subspace $Z$ of $X$ is an \emph{$L$-summand} if $X=Z\oplus_1 W$
for some closed subspace $W$ of $X$, where $\oplus_1$ denotes the
$\ell_1$-sum. A closed subspace $Y$ of a Banach space $X$ is said to be an
\emph{$M$-ideal} of $X$ if the annihilator $Y^\perp$ of $Y$ is an
$L$-summand of $X^*$. We refer the reader to the monograph by P.~Harmand,
D.~Werner and W.~Werner \cite{HWW} for background on $L$-summands and
$M$-ideals.

\section{The tool: numerical range and isometries}
The main tool we are using in the paper is the relationship between
isometries and the numerical range of operators on Banach spaces, a
concept independently introduced by F.~Bauer \cite{Bauer} and G.~Lumer
\cite{Lumer} in the sixties to extend the classical field of values of
matrices (O.~Toeplitz, 1918 \cite{Toe}). We refer the reader to the
monographs by F.~Bonsall and J.~Duncan \cite{B-D1,B-D2} from the seventies
for background and more information. Let $X$ be a real or complex Banach
space. The \emph{numerical range} of an operator $T\in L(X)$ is the subset
$V(T)$ of the scalar field defined by
$$
V(T)=\{x^*(Tx)\ : \ x\in S_X,\ x^*\in S_{X^*},\ x^*(x)=1\}.
$$
This is Bauer's definition, while Lumer's one depends upon the election of
a compatible semi-inner product on the space (a notion also introduced by
Lumer \cite{Lumer}). But concerning our applications, both of them are
equivalent in the sense that they have the same closed convex hull. The
\emph{numerical radius} is the seminorm defined by
$$
v(T)=\sup\{|\lambda|\ : \ \lambda\in V(T)\}
$$
for every $T\in L(X)$. It is clear that $v$ is continuous; actually,
$v(T)\leqslant \|T\|$ for every $T\in L(X)$. We say that $T\in L(X)$ is
\emph{skew-hermitian} if $\re V(T)= \{0\}$; we write $\mathcal{A}(X)$ for
the closed subspace of $L(X)$ consisting of all skew-hermitian operators
on $X$, which is called the \emph{Lie algebra} of $X$ by H.~Rosenthal
\cite{Rosenthal}. In the real case, $T\in \mathcal{A}(X)$ if and only if
$v(T)=0$.

Let us give a clarifying example. If $H$ is a $n$-dimensional Hilbert
space, it is easy to check that $\mathcal{A}(H)$ is the space of
skew-symmetric operators on $H$ (i.e.\ $T^*=-T$ in the Hilbert space
sense), so it identifies with the space of skew-symmetric matrices. It is
a classical result from the theory of linear algebra that a $n\times n$
matrix $A$ is skew-symmetric if and only if $\exp(\rho A)$ is an
orthogonal matrix for every $\rho\in\R$ (see
\cite[Corollary~8.5.10]{Artin-Algebra} for instance). The same is true for
an infinite-dimensional Hilbert space by just replacing orthogonal
matrices by unitary operators (i.e.\ surjective isometries).

The above fact extends to general Banach spaces. Indeed, for an arbitrary
Banach space $X$ and an operator $T\in L(X)$, by making use of the
\emph{exponential formula} \cite[Theorem~3.4]{B-D1}
\begin{equation*}
\sup \re V(T)=\lim_{\beta\downarrow 0}\frac{\|\Id +
\beta\,T\|-1}{\beta}=\sup_{\alpha>0}\ \frac{\log\|\exp(\alpha\,
T)\|}{\alpha}\, ,
\end{equation*}
the following known result is easy to prove.

\begin{prop}[{\rm \cite[Theorem~1.4]{Rosenthal}}]\label{prop:Rosenthal}
Let $X$ be a real or complex Banach space and $T\in L(X)$. Then, the
following are equivalent.
\begin{enumerate}
\item[$(i)$] $T$ is skew-hermitian.
\item[$(ii)$] $\|\exp(\rho\,T)\|\leqslant 1$ for every $\rho\in \R$.
\item[$(iii)$] $\{\exp(\rho\,T)\, : \, \rho\in \R\}\subset \iso(X)$, i.e.\ $T$
is the generator of a uniformly continuous one-parameter subgroup of
isometries.
\item[$(iv)$] $T\in \mathcal{T}(\iso(X),\Id)$.
\end{enumerate}
Therefore, $\mathcal{A}(X)$ coincides with
$\mathcal{T}(\iso(X),\Id)$ and with the set of generators of
uniformly continuous one-parameter subgroups of isometries.
\end{prop}

In the real case, the above result leads us to consider when the numerical
radius is a norm on the space of operators. A related concept to this fact
is the \emph{numerical index} of a (real or complex) Banach space $X$,
introduced by G.~Lumer in 1968, which is the constant $n(X)$ defined by
$$
n(X)=\inf\{v(T)\ : \ T\in L(X),\ \|T\|=1\}
$$
or, equivalently, the greatest constant $k\geqslant 0$ such that
$k\|T\|\leqslant v(T)$ for every $T\in L(X)$. Let us mention here a couple
of facts concerning the numerical index which will be relevant to our
discussion. For more information, recent results and open problems, we
refer the reader to the very recent survey \cite{survey-racsam} and
references therein. First, real or complex $C(K)$ and $L_1(\mu)$ spaces
have numerical index~$1$, and the numerical index of a Hilbert space of
dimension greater than one is $0$ in the real case and $1/2$ in the
complex case. The set of values of the numerical index is given by the
following equalities:
\begin{eqnarray*}
\{n(X)\ : \ X \text{ complex Banach space} \ \}&=&[\e^{-1},1], \\
\{n(X)\ : \ X \text{ real Banach space} \ \}&=&[0,1].
\end{eqnarray*}
Finally, since $v(T)=v(T^*)$ for every $T\in L(X)$ (this can be proved by
making use of the exponential formula, for example), it follows that
$n(X^*)\leqslant n(X)$ for every Banach space $X$. Very recently,
K.~Boyko, V.~Kadets, D.~Werner and the author proved that the reversed
inequality does not always hold \cite[Example~3.1]{BKMW}, answering in the
negative way a question which had been latent since the beginning of the
seventies.

Let us observe that, for a \emph{real} Banach space $X$, as a consequence
of Proposition~\ref{prop:Rosenthal}, if $\mathcal{T}(\iso(X),\Id)$ is
non-trivial, then $n(X)=0$. Let us state this result for further
reference.

\begin{prop}\label{pro:npositive}
Let $X$ be a real Banach space with $n(X)>0$. Then,
$\mathcal{T}(\iso(X),\Id)$ reduces to zero.
\end{prop}

In the finite-dimensional case, the above implication reverses
and, actually, the numerical index zero characterizes those
finite-dimensional real Banach spaces with infinitely many
isometries \cite[Theorem~3.8]{Rosenthal}. We refer the reader
to the just cited \cite{Rosenthal} and to
\cite{MarMeRo,Ros-Pacific} for further results on
finite-dimensional spaces with infinitely many isometries. In
the infinite-dimensional case, the situation is different and
it is possible to find a real Banach space $X$ such that
$n(X)=0$ but $\mathcal{T}(\iso(X),\Id)=\{0\}$, i.e.\ the
numerical radius it is a (necessarily non-complete) norm on
$L(X)$ which is not equivalent to the usual one.

We will also use another concept related to the numerical range: the
so-called Daugavet equation. A bounded linear operator $T$ on a Banach
space $X$ is said to satisfy the \emph{Daugavet equation} if
\begin{equation*}\label{DE}\tag{DE}
\|\Id + T\|= 1 + \|T\|.
\end{equation*}
This norm equality appears for the first time in a 1963 paper
by I.~Daugavet \cite{Dau}, where it was proved that every
compact operator on $C[0,1]$ satisfies it. Following
\cite{KSSW0,KSSW}, we say that a Banach space $X$ has the
\emph{Daugavet property} if every rank-one operator $T\in L(X)$
satisfies \eqref{DE}. In such a case, every operator on $X$ not
fixing a copy of $\ell_1$ also satisfies \eqref{DE} \cite{Shv};
in particular, this happens to every compact or weakly compact
operator on $X$ \cite{KSSW}. Examples of spaces with the
Daugavet property are $C(K)$ when the compact space $K$ is
perfect, and $L_1(\mu)$ when the positive measure $\mu$ is
atomless. An introduction to the Daugavet property can be found
in the books by Y.~Abramovich and C.~Aliprantis
\cite{AbrAli-1,AbrAli-2} and the state-of-the-art can be found
in the survey paper by D.~Werner \cite{WerSur}; for more recent
results we refer the reader to \cite{BM,BKW,I-K-W,KadSheWer}
and references therein.

The relation between the Daugavet equation and the numerical range is
given as follows \cite{D-Mc-P-W}. Given a Banach space $X$ and $T\in
L(X)$,
\begin{center}
$T$ satisfies \eqref{DE} $\quad \Longleftrightarrow \quad$ $\sup \re
V(T)=\|T\|$.
\end{center}
The following result is an straightforward consequence of this
fact.

\begin{prop}\label{pro:Daug-circle}
Let $X$ be a real or complex Banach space and $T\in L(X)$. If $\lambda\,T$
satisfies \eqref{DE}, then $\|T\| \in \lambda\,\overline{V(T)}$. In
particular, if $X$ has the Daugavet property, then every operator $T\in
L(X)$ which does not fix a copy of $\ell_1$ satisfies
$$
\|T\|\,\T \subset \overline{V(T)}.
$$
\end{prop}

We finish this section collecting some easy results concerning
$L$-summands of Banach spaces, numerical ranges, and isometries. We
include a proof for the sake of completeness.

\begin{prop}\label{prop:Lsummand}
Let $X$ be a real or complex Banach space and suppose that $X=Y\oplus_1 Z$
for closed subspaces $Y$ and $Z$.
\begin{enumerate}
\item[(a)] Given an operator $S\in L(Y)$, the operator $T\in L(X)$ defined by
$$
T(y,z)=(Sy,0) \qquad \big(y\in Y,\ z\in Z\big)
$$
satisfies $\|T\|=\|S\|$ and $V(T)\subset [0,1]\,V(S)$.
\item[(b)] For every $S\in \iso(Y)$, the operator
$$
T(y,z)=(Sy,z) \qquad \big(y\in Y,\ z\in Z\big)
$$
belongs to $\iso(X)$.
\end{enumerate}
\end{prop}

\begin{proof} (a). It is clear that $\|T\|=\|S\|$. On the other hand, given
$$
\lambda=(y^*,z^*)\big(T(y,z)\big)=y^*(Sy)\in V(T),
$$
where $(y,z)\in S_X$ and $(y^*,z^*)\in S_{X^*}$ with $(y^*,z^*)(y,z)=1$,
we have
$$
1=(y^*,z^*)(y,z)=y^*(y)+z^*(z)\leqslant \|y^*\|\|y\|+\|z^*\|\|z\|\leqslant
\|y\|+\|z\|=1.
$$
We deduce that $y^*(y)=\|y^*\|\,\|y\|$. If $\|y^*\|\,\|y\|=0$, then
$\lambda=0\in [0,1]\,V(S)$. Otherwise,
\begin{equation}
\lambda=\|y^*\|\|y\|\,\frac{y^*}{\|y^*\|}\left(S\frac{y}{\|y\|}\right)\in
[0,1]\,V(S).
\end{equation}
(b). For $(y,z)\in X$, we have
\begin{equation*}
\|T(y,z)\|=\|(Sy,z)\|=\|Sy\|+\|z\|=\|y\|+\|z\|=\|(y,z)\|.\qedhere
\end{equation*}
\end{proof}

\section{The construction}
Our aim here is to construct closed subspaces of $C[0,1]$, which share
some properties with it, but such that their duals could be extremely
different from $C[0,1]^*$. The idea of our construction is to squeeze the
one that was given in \cite[Examples 3.1 and 3.2]{BKMW} to show that the
numerical index of the dual of a Banach space can be different than the
numerical index of the space. Let us comment that all the results in this
section are valid in the real and in the complex case. We need one
definition.

\begin{definition}[{\rm \cite[Definition~2.3]{BKMW}}]\label{def-crich}
Let $K$ be a compact Hausdorff space. A closed subspace $X$ of $C(K)$ is
said to be \emph{C-rich} if for every nonempty open subset $U$ of $K$ and
every $\eps > 0$, there is a positive continuous function $h$ of norm $1$
with support inside $U$ such that the distance from $h$ to $X$ is less
than~$\eps$.
\end{definition}

Our interest in C-rich subspaces of $C(K)$ is that they inherit some
geometric properties from $C(K)$, as the following result summarizes.

\begin{prop}[{\rm \cite[Theorem~2.4]{BKMW} and \cite[Theorem~3.2]{Ka-Po}}]\label{prop:c-rich-numindex1}
Let $K$ be a perfect compact Hausdorff space and let $X$ be a C-rich
subspace of $C(K)$. Then $n(X)=1$ and $X$ has the Daugavet property.
\end{prop}

We are now able to present the main result of the paper.

\begin{theorem}\label{th:construction}
Let $E$ be a separable Banach space. Then, there is a C-rich subspace
$X(E)$ of $C[0,1]$ such that $X(E)^*$ contains (an isometrically
isomorphic copy of) $E^*$ as an $L$-summand.
\end{theorem}

\begin{proof}
By the Banach-Mazur Theorem, we may consider $E$ as a closed subspace of
$C(\Delta)$, where $\Delta$ denotes the Cantor middle third set viewed as
a subspace of $[0,1]$. We write $P:C[0,1]\longrightarrow C(\Delta)$ for
the restriction operator, i.e.\
$$
\big[P(f)\big](t)=f(t) \qquad \big(t\in\Delta,\ f\in C[0,1]\big).
$$
We define the closed subspaces $X(E)$ and $Y$ of $C[0,1]$ by
$$
X(E)=\left\{f\in C[0,1]\ : \ P(f)\in E\right\}, \qquad Y=\ker P.
$$
Since $[0,1]\setminus \Delta$ is open and dense in $[0,1]$, it is
immediate to show that $X(E)$ is C-rich in $C[0,1]$. Indeed, for every
nonempty open subset $U$ of $[0,1]$, we consider the nonempty open subset
$V=U\cap\, \big([0,1]\setminus \Delta\big)$ and we take a norm-one
continuous function $h:[0,1]\longrightarrow [0,1]$ whose support is
contained in $V$. Therefore, $h$ belongs to $Y\subseteq X(E)$, and the
support of $h$ is contained in $U$.

Since $Y$ is an $M$-ideal in $C[0,1]$ (see \cite[Example~I.1.4(a)]{HWW}),
it is a fortiori an $M$-ideal in $X(E)$ by \cite[Proposition~I.1.17]{HWW},
meaning that $Y^\perp\equiv \left(X(E)/Y\right)^*$ is an $L$-summand of
$X(E)^*$.

It only remains to prove that $X(E)/Y$ is isometrically
isomorphic to $E$. To do so, we define the operator
$\Phi:X(E)\longrightarrow E$ given by $\Phi(f)=P(f)$ for every
$f\in X(E)$. Then $\Phi$ is well defined, $\|\Phi\|\leqslant
1$, and $\ker \Phi=Y$. To see that the canonical quotient
operator $\widetilde{\Phi}:X(E)/Y\longrightarrow E$ is a
surjective isometry, it suffices to show that
$$
\Phi\big(\{f\in X(E)\ : \ \|f\|<1\}\big) = \{g\in E\ : \ \|g\|<1\}.
$$
Indeed, the left-hand side is contained in the right-hand side
since $\|\Phi\|\leqslant 1$. On the other hand, for every $g\in
E\subset C(\Delta)$ with $\|g\|<1$, we consider any isometric
extension $f\in C[0,1]$ (it is easy to construct it by just
considering an affine extension, see
\cite[p.~18]{Albiac-Kalton} for instance). It is clear that
$f\in X(E)$ with $\|f\|=\|g\|<1$ and that $\Phi(f)=g$.
\end{proof}

\begin{remarks}\label{remark:l1mueneldual}$ $
\begin{enumerate}
\item[(a)] Let us observe that the spaces $X(E)$ has a
    strong version of C-richness which can be read as the
    validity of the Urysohn lemma in $X(E)$. Namely,
    {\slshape for every nonempty open subset $U$ of
    $[0,1]$, there is a non-null positive continuous
    function $h\in X(E)$ whose support is contained in
    $U$.}
\item[(b)] Also, following the proof of the theorem, it is easy to check what we have
actually proved is that {\slshape $X(E)^*\equiv E^*\oplus_1 L_1(\mu)$ for
a suitable localizable positive measure $\mu$.\ } Indeed, we have shown
that $Y$ is an $M$-ideal in $X(E)$ and $Y^\perp=(X(E)/Y)^*\equiv E^*$. By
\cite[Remark~I.1.13]{HWW}, one gets $X(E)^*\equiv E^* \oplus_1 Y^*$. On
the other hand, $Y$ is an $M$-ideal in $C[0,1]$ and so, $Y^*$ is
(isometrically isomorphic to) an $L_1(\mu)$ space. To this end, one may
make use of \cite[Example~I.1.6(a)]{HWW} and of the fact that $C[0,1]^*$
is isometric to an $L_1(\nu)$ space for some localizable positive measure
$\nu$.
\item[(c)] The above observation leads us to give a direct and simple
proof of the fact that $n(X(E))=1$ for every $E$, without calling
Proposition~\ref{prop:c-rich-numindex1}. Indeed, in the identification
$X(E)^*\equiv E^*\oplus_1 Y^*$, the evaluation functionals
$$
A=\{\delta_t\ : \ t\in [0,1]\setminus \Delta\}
$$
(where, as usual, $\delta_t(f)=f(t)$) belong to $Y^*\equiv L_1(\mu)$ (see
\cite[Proposition~I.1.12]{HWW}). Being $[0,1]\setminus \Delta$ dense in
$[0,1]$, it follows that $B_{X(E)^*}$ is the $w^*$-closed convex hull of
$A$. On the other hand, every extreme point of the unit ball of
$X(E)^{**}\equiv E^{**}\oplus_\infty L_\infty(\mu)$ is of the form
$(e^{**},h)$ where $e^{**}$ is extreme in $B_{E^{**}}$ and $h$ is extreme
in $B_{L_\infty(\mu)}$. It implies that
$$
|x^{**}(a)|=1 \qquad \big(a\in A,\ x^{**} \text{ extreme point of }
B_{X^{**}}\big).
$$
Now, by just using that $v(T)=v(T^*)$ for every $T\in L(X(E))$, it is easy
to check that $n(X(E))=1$ (see \cite[Proposition~6]{survey-racsam} for
example).
\end{enumerate}
\end{remarks}

The construction in Theorem~\ref{th:construction} can be easily extended
to the general case in which $E$ is not separable by just replacing
$[0,1]$ by a convenient perfect compact space $K$. The main difference is
that, obviously, there is no universal such a $K$.

\begin{prop} Let $E$ be a Banach space. Then there is a
perfect Hausdorff compact space $K$ and a C-rich subspace $X(E)$ of $C(K)$
such that $E^*$ is an $L$-summand of $X(E)^*$.
\end{prop}

\begin{proof} We consider $E$ as a closed subspace of $C\big((B_{E^*},w^*)\big)$,
we write $K$ for the perfect compact space $(B_{E^*},w^*)\times[0,1]$,
$P:C(K)\longrightarrow C\big((B_{E^*},w^*)\big)$ for the operator
$$
\big[P(f)\big](t)=f(t,0) \qquad \big(t\in B_{E^*},\ f\in C(K)\big),
$$
and we consider the space
$$
X(E)=\left\{f\in C(K)\ : \ P(f)\in E\right\}.
$$
Now, it is easy to adapt the proof of Theorem~\ref{th:construction} to
this situation.
\end{proof}

\section{The examples}
Our aim here is to use Theorem~\ref{th:construction} with some particular
spaces $E$ to produce some interesting examples. The first one is the
promised space whose group of isometries is much smaller than the one of
its dual.

\begin{example}\label{main-example}
{\slshape The real Banach space $X(\ell_2)$ produced in
Theorem~\ref{th:construction} satisfies that $\iso(X(\ell_2))$
does not contain any non-trivial uniformly continuous
one-parameter subgroup, while $\iso(X(\ell_2)^*)$ contains
infinitely many uniformly continuous one-parameter subgroups.
Equivalently, $\mathcal{T}(\iso(X(\ell_2)),\Id)=\{0\}$ but
$\mathcal{T}(\iso(X(\ell_2)^*),\Id)$ is infinite-dimensional.\
}
\end{example}

\begin{proof}
Being $n(X(\ell_2))=1$ by
Proposition~\ref{prop:c-rich-numindex1}, the tangent space at
$\Id$ of the group $\iso\big(X(\ell_2)\big)$ is null by
Proposition~\ref{pro:npositive}. On the other hand,
Proposition~\ref{prop:Lsummand}.b gives us that the group of
isometries of $X(\ell_2)^*$ contains $\iso(\ell_2)$ as a
subgroup, and so $\mathcal{T}\big(X(\ell_2)^*,\Id\big)$ is
infinite-dimensional.
\end{proof}

The above example gives us, in particular, with a real Banach
space with numerical index~$1$ such that its dual has numerical
index~$0$. An example of this kind was given in
\cite[Example~3.2.a]{BKMW} as a $c_0$-sum of spaces whose duals
have positive numerical index. Therefore, the dual of that
example does not contain any non-null skew-hermitian operator
(see \cite[Example~3.b]{M-P} for a proof). Thus, the existence
of a space like the one given in Example~\ref{main-example} is
not contained in \cite[Example~3.2.a]{BKMW}. On the other hand,
the new construction can be used to give the following
improvement of the examples given in \cite{BKMW}: the dual of a
Banach space with numerical index~$1$ may have any possible
value of the numerical index.

\begin{prop}
\begin{eqnarray*}
\{n(X^*)\ : \ X \text{ complex Banach space with $n(X)=1$} \ \}&=&[\e^{-1},1], \\
\{n(X^*)\ : \ X \text{ real Banach space with $n(X)=1$} \ \}&=&[0,1].
\end{eqnarray*}
\end{prop}

\begin{proof} Indeed, just take two-dimensional spaces $E$ with any
possible value of the numerical index (see \cite{D-Mc-P-W}) and consider
$X(E)$. Then, by Theorem~\ref{th:construction} and
Remark~\ref{remark:l1mueneldual}.b, $n(X(E))=1$ while $X(E)^*=E^*\oplus_1
L_1(\mu)$ for a suitable measure $\mu$. Now, \cite[Proposition~1]{M-P}
gives
\begin{equation*}
n(X(E)^*)=\min\{n(E^*),\ L_1(\mu)\}=\min\{n(E^*),1\}=n(E^*)=n(E).\qedhere
\end{equation*}
\end{proof}

In a 1977 paper \cite{Hur}, T.~Huruya determined the numerical index of a
(complex) $C^*$-algebra. Part of the proof was recently clarified by
A.~Kaidi, A.~Morales, and A.~Rodr\'{\i}guez-Palacios in \cite{K-M-RP}, where
the result is extended to preduals of von Neumann algebras. Namely, the
numerical index of a $C^*$-algebra is equal to $1/2$ when it is not
commutative and $1$ when it is commutative, and the numerical index of the
predual of a von Neumann algebra coincides with the numerical index of the
algebra. Therefore, if $X$ is a $C^*$-algebra or the predual of a von
Neumann algebra, then $n(X)=n(X^*)$, and the same is true for all the
successive duals of $X$. The following example shows that we can not
extend this result to successive preduals.

\begin{example} {\slshape There exists a Banach space $X$ such
that $X^{**}$ is (isometrically isomorphic to) a $C^*$-algebra, $n(X)=1$
and $n(X^*)=1/2$.\ }
\end{example}

\begin{proof} We consider $E=K(\ell_2)$, the space of compact
linear operators on $\ell_2$. Then, the space $X(E)$ given in
Theorem~\ref{th:construction} has numerical index~$1$ and
Remark~\ref{remark:l1mueneldual}.b gives us that
$$
X(E)^*\equiv K(\ell_2)^*\oplus_1 L_1(\mu)
$$
for a suitable positive localizable measure $\mu$. Then,
$n(X(E)^*)=n(K(\ell_2)^*)=1/2$ and $X(E)^{**}$ is isometrically isomorphic
to the $C^*$-algebra $L(\ell_2)\oplus_\infty L_\infty(\mu)$.
\end{proof}

Let us observe that the adjoint of a skew-hermitian operator on a Banach
space $X$ is also skew-hermitian, and Example~\ref{main-example} shows
that there might be skew-hermitian operators on $X^*$ which are not
$w^*$-continuous. The next two examples show that the same is true for
hermitian operators and dissipative operators.

Let $X$ be a complex Banach space. An operator $T\in L(X)$ is said to be
\emph{hermitian} if $V(T)\subset \R$ (i.e.\ the operator $i\,T$ is
skew-hermitian), equivalently (see Proposition~\ref{prop:Rosenthal}), if
$\exp(i\rho\, T)\in \iso(X)$ for every $\rho\in \R$. Hermitian operators
have been deeply studied and many results on Banach algebras depend on
them; for instant, the Vidav-Palmer characterization of $C^*$-algebras. We
refer to \cite{B-D1,B-D2,Lumer} for more information.

\begin{example}
{\slshape Let us consider the complex space $X(\ell_2)$ produced in
Theorem~\ref{th:construction}. Then, every non-null hermitian operator $T$
on $X$ fixes a copy of $\ell_1$ whereas $X(\ell_2)^*$ has an
infinite-dimensional real subspace of finite-rank hermitian operators.\ }
\end{example}

\begin{proof}
The space $X(\ell_2)$ has the Daugavet property by
Proposition~\ref{prop:c-rich-numindex1}, and so
Proposition~\ref{pro:Daug-circle} gives us $\|T\|\,\T \subseteq
\overline{V(T)}$ for every operator $T$ which does not fix a
copy of $\ell_1$. This implies that $T$ is not hermitian. On
the other hand, being $\ell_2$ an $L$-summand of $X(\ell_2)^*$,
every finite-dimensional orthogonal projection on $\ell_2$
(which is clearly hermitian) canonically extend to a
finite-rank hermitian operator on $X(\ell_2)^*$ by
Proposition~\ref{prop:Lsummand}.
\end{proof}

Our last example deals with dissipative operators, a concept translated
from the Hilbert space setting to general Banach spaces by G.~Lumer and
R.~Phillips \cite{Lumer-Phillips} in 1961. Let $X$ be a real or complex
Banach space. An operator $T\in L(X)$ is said to be \emph{dissipative} if
$\re V(T)\subset \R^-_0$ or, equivalently, if $\|\exp(\rho T)\|\leqslant
1$ for every $\rho\in \R^+$ (i.e.\ $T$ is the generator of a one-parameter
semigroup of contractions). We refer to \cite{B-D1,B-D2,Lumer-Phillips}
for more information and background.

\begin{example}
{\slshape Let us consider the real or complex space $X(\ell_2)$
produced in Theorem~\ref{th:construction}. Then, every non-null
dissipative operator $T$ on $X$ fixes a copy of $\ell_1$
whereas $X(\ell_2)^*$ has infinitely many linear independent
finite-rank dissipative operators.\ }
\end{example}

\begin{proof}
The space $X(\ell_2)$ has the Daugavet property by
Proposition~\ref{prop:c-rich-numindex1}, and so
Proposition~\ref{pro:Daug-circle} gives us $\|T\|\,\T \subseteq
\overline{V(T)}$ for every operator $T$ which does not fix a
copy of $\ell_1$. This implies that $T$ is not dissipative. On
the other hand, being $\ell_2$ an $L$-summand of $X(\ell_2)^*$,
the opposite of every finite-dimensional orthogonal projection
on $\ell_2$ (which is clearly dissipative) canonically extend
by zero to a finite-rank dissipative operator on $X(\ell_2)^*$.
\end{proof}

\noindent \textbf{Acknowledgments.\ } The starting point of this paper was
a question asked to me by Rafael Pay\'{a} during the Madrid ICM-2006. I am
indebted to him. I also thank Javier Mer\'{\i} and Armando Villena for fruitful
conversations about the subject of this note.

\end{document}